\documentclass[12pt]{amsart}
\usepackage{amssymb}

\swapnumbers

\newtheorem{thm}{Theorem}[section]
\newtheorem{lem}[thm]{Lemma}
\newtheorem{cor}[thm]{Corollary}
\newtheorem{prop}[thm]{Proposition}
\newtheorem{example}[thm]{Example}

\newtheorem*{prob*}{Open problem}

\theoremstyle{definition}
\newtheorem{defi}[thm]{Definition}

\theoremstyle{remark}
\newtheorem{rem}[thm]{Remark}
\newtheorem*{rem*}{Remark}

\newcommand{\kringel}{\mathbin{\raise1pt\hbox{$\scriptstyle\circ$}}} 
\newcommand{\pkt}{\mathbin{\raise0pt\hbox{$\scriptstyle\bullet$}}}
\newcommand{\ph}{\phantom{0000}}
\newcommand{\lf}{\lfloor}
\newcommand{\rf}{\rfloor}

\newcommand{\HH}{\mathbb{H}}

\newcommand{\Lg}{\mathfrak{g}}

\newcommand{\abs}[1]{\lvert#1\rvert}

\newcommand{\al}{\alpha}

\newcommand{\de}{\delta}

\newcommand{\ra}{\rightarrow}  

\renewcommand{\phi}{\varphi}

\begin{document}

\title[Binomial sums of Partition functions]{Estimates on binomial sums 
of partition functions} 
%  Die Kurzfassung kommt oben ueber die Seiten, sie steht in eckigen Klammern
%  Auch Autorennamen koennen eine Kurzfassung haben

\author[D. Burde]{Dietrich Burde}
\address{Mathematisches Institut\\
  Heinrich-Heine-Universit\"at\\
  Universit\"atsstr. 1\\
  40225 D\"us\-sel\-dorf\\
  Germany}
\email{dietrich@math.uni-duesseldorf.de}


\begin{abstract}
Let $p(n)$ denote the partition function and define 
$p(n,k)=\sum_{j=0}^{k}\binom{n-j}{k-j}p(j)$ where $p(0)=1$.
We prove that $p(n,k)$ is unimodal and satisfies 
$p(n,k) < \frac{2.825}{\sqrt{n}}\, 2^n $ for fixed $n\ge 1$ and all 
$1\le k\le n$.
This result has an interesting application:
the minimal dimension of a faithful module for a $k$-step nilpotent Lie
algebra of dimension $n$ is bounded by $p(n,k)$ and hence by
$\frac{3}{\sqrt{n}}\, 2^n $, independently of $k$.
So far only the bound $n^{n-1}$ was known.
We will also prove that
$p(n,n-1)<\sqrt{n}\exp(\pi\sqrt{2n/3})$ for $n\ge 1$ and
$p(n-1,n-1)<\exp (\pi\sqrt{2n/3} )$.
\end{abstract}

\maketitle

\section{Introduction}

Let $\Lg$ be a Lie algebra of dimension $n$ over a field $K$ of
characteristic zero. An invariant of $\Lg$ is defined by
$$\mu (\Lg):=\min \{\dim M \mid M \text{ is a faithful $\Lg$--module}\}$$   
Ado's theorem asserts that $\mu(\Lg)$ is finite. Following the details of
the proof we see that $\mu(\Lg)\le f(n)$ for a function $f$ only
depending on $n$.
It is an open problem to determine good upper bounds for $f(n)$
valid for a given class of Lie algebras of dimension $n$.
Interest for such a refinement of Ado's theorem  
comes from a question of Milnor on fundamental groups of complete
affine manifolds \cite{M2}.
The existence of left-invariant affine structures on a Lie group $G$
of dimension $n$ implies $\mu(\Lg)\le n+1$ for its Lie algebra $\Lg$. 
It is known that there exist nilpotent Lie algebras which do not
satisfy this bound \cite{BG}. It is however difficult to prove good bounds
for $\mu (\Lg)$ only depending on $\dim \Lg$.
In $1937$ Birkhoff \cite{BIR} proved $\mu(\Lg)\le 1+n+n^2+\dots+n^{k+1}$
for all nilpotent Lie algebras $\Lg$ of dimension $n$ and 
nilpotency class $k$.
His construction used the universal enveloping algebra of $\Lg$.
In $1969$ this method was slightly improved by Reed \cite{REE} who
proved  $\mu(\Lg)\le 1+n^k$. 
That yields the bound $\mu(\Lg)\le 1+n^{n-1}$
only depending on $n$.
We have improved the bound in \cite{BU5} as follows:
\begin{thm}\label{th1}
Let $\Lg$ be a nilpotent Lie algebra of dimension $n$ and nilpotency class
$k$. Denote by $p(n)$ the number of partitions of 
$n$ into positive integers with $p(0)=1$ and set  
$$p(n,k)=\sum_{j=0}^{k}\binom{n-j}{k-j}p(j).$$
Then $\mu(\Lg)\le p(n,k)$.
\end{thm} 

The aim of this paper is to study the function $p(n,k)$ and to give
upper bounds for it. We will show the following: 

\begin{thm}\label{th2}
The function $p(n,k)$ is unimodal for fixed $n\ge 4$. More precisely we
have with $k(n)=\lfloor \frac{n+3}{2} \rfloor$
\begin{gather*}
p(n,1)<p(n,2)<\dots <p(n,k(n)-1)< p(n,k(n)),\\
p(n,k(n)) > p(n,k(n)+1)> \dots > p(n,n-1)>p(n,n).
\end{gather*}    
\end{thm}
\begin{thm}\label{th3}
There is the following estimate for $p(n,k)$:
\begin{align*}
p(n,k) & < \frac{2.825}{\sqrt{n}}\, 2^n \;\text{ for fixed $n\ge 1$ and all }
1\le k\le n
\end{align*}
\end{thm}

\begin{cor}
Let $\Lg$ be a nilpotent Lie algebra of dimension $n$. Then
$$\mu(\Lg)<\frac{3}{\sqrt{n}}\, 2^n$$
\end{cor} 

A nilpotent Lie algebra $\Lg$ of dimension $n$ and nilpotency class $k$
is called {\it filiform} if $k=n-1$. 
In that case the estimate for $\mu (\Lg)$ can be improved. In fact it holds
$\mu(\Lg)\le 1+p(n-2,n-2)$ which was the motivation to prove the
following propositions:

\begin{prop}\label{pr1}
Let $\al=\sqrt{\frac{2}{3}}\pi$. Then 
$$p(n-1,n-1)<e^{\al \sqrt{n}}\quad \text{for all} \;\; n\ge 1.$$
\end{prop}

\begin{prop}\label{pr2} 
Let $\al=\sqrt{\frac{2}{3}}\pi$. Then 
$$p(n,n-1)<\sqrt{n}e^{\al \sqrt{n}} \quad \text{for all} \;\; n\ge 1.$$ 
\end{prop}

\begin{rem}
If $k,n \ra \infty$ with $\frac{k}{n}\le 1-\de$ for some
fixed $\de>0$ then one has asymptotically
\begin{equation*}
p(n,k)\sim \binom{n}{k}\prod_{j=1}^{\infty} \frac{1}{1-(\frac{k}{n})^j}.
\end{equation*}
For $k/n=1/2$ the infinite product is approximately
$3.4627466194550636$.
The theorem shows that $\mu(\Lg)\le p(n,k)$ is a better estimate than 
$\mu(\Lg)\le 1+n^k$, especially if $k$ is not small in 
comparison to $n$. As for a bound for $\mu(\Lg)$ independent of $k$,
the corollary yields a better one than $n^{n-1}$.
Note that some of the estimates on $p(n,k)$ have been stated in 
\cite{BU5}, where the proof of Lemma $5$ is not complete.
In fact, the upper bound given there for $p(n,n-1)$ depends on a strong
upper bound for $p(n)$ itself, which so far is not proved. Using
the known upper bound for $p(n)$ in \cite{APO} however it is not difficult
to prove the above estimates.
\end{rem}  

We have included a table which
shows the values for $p(k)$ and $p(n,k)$ for
$n=50$ and $1\le k\le 50$. I thank Michael Stoll for helpful discussions.

\hspace*{1cm}
\begin{center}
\begin{tabular}{|c|c|l|}
\hline
\phantom{00} $k$ \phantom{00}& \phantom{00} $p(k)$ \phantom{00} & 
\phantom{00} $p(50,k)$ \\
\hline\hline
$1$  & $1$  & \ph $51$       \\ 
$2$  & $2$  & \ph $1276$     \\ 
$3$  & $3$  & \ph $20875$    \\ 
$4$  & $5$  & \ph $251126$   \\ 
$5$  & $7$  & \ph $2368708$  \\ 
$6$  & $11$ & \ph $18240890$ \\
$7$  & $15$ & \ph $117911248$ \\ 
$8$  & $22$ & \ph $652850403$ \\ 
$9$  & $30$ & \ph $3143939547$ \\ 
$10$ & $42$ & \ph $13327191287$ \\ 
$11$ & $56$   & \ph $50207862055$  \\
$12$ & $77$   & \ph $169422173829$  \\
$13$ & $101$   & \ph $515401493777$  \\
$14$ & $135$   & \ph $1421191021907$  \\
$15$ & $176$   & \ph $3568459118188$  \\
$16$ & $231$   & \ph $8190773240690$  \\
$17$ & $297$   & \ph $17243902126004$  \\
$18$ & $385$   & \ph $33393294003697$  \\
$19$ & $490$   & \ph $59630690096752$  \\
$20$ & $627$   & \ph $98399515067097$  \\
$21$ & $792$   & \ph $150323197512416$  \\
$22$ & $1002$   & \ph $212938456376977$  \\
$23$ & $1255$   & \ph $280067870621181$  \\
$24$ & $1575$   & \ph $342413939297475$  \\
$25$ & $1958$   & \ph $389526824102747$  \\
$26$ & $2436$   & \ph $412637434996367$ \ph \\
$27$ & $3010$   & \ph $407312833046180$  \\
$28$ & $3718$   & \ph $374834739612319$  \\
$29$ & $4565$   & \ph $321717177399531$  \\
$30$ & $5604$   & \ph $257604118720316$  \\
$31$ & $6842$   & \ph $192465300826581$  \\
$32$ & $8349$   & \ph $134186828954271$  \\
$33$ & $10143$   & \ph $87302345518136$  \\
$34$ & $12310$   & \ph $52999252173708$  \\
$35$ & $14883$   & \ph $30018139013576$  \\
$36$ & $17977$   & \ph $15859467681399$  \\
$37$ & $21637$   & \ph $7814276022624$  \\
$38$ & $26015$   & \ph $3589870410395$  \\
$39$ & $31185$   & \ph $1537270615509$  \\
$40$ & $37338$   & \ph $613479208559$  \\
$41$ & $44583$   & \ph $228106170152$  \\
$42$ & $53174$   & \ph $79012160892$  \\
$43$ & $63261$   & \ph $25493798901$  \\
$44$ & $75175$   & \ph $7662394094$  \\
$45$ & $89134$   & \ph $2145558341$  \\
$46$ & $105558$   & \ph $559858427$  \\
$47$ & $124754$   & \ph $136194920$  \\
$48$ & $147273$   & \ph $30906004$  \\
$49$ & $173525$   & \ph $6547151$  \\
$50$ & $204226$   & \ph $1295971$  \\
\hline
\end{tabular} 
\end{center}

\section{Unimodality}

\begin{defi}
Let $f$ be a sequence and define
\begin{equation*}
F(n,\ell):=\sum_{j=0}^n \binom{n-j}{\ell}f(j)
\end{equation*}
for $0\le \ell\le n$, where the binomial coefficient is understood
to be zero if $n-j<\ell$. 
Then $F(n,\ell)$ is called {\it unimodal}, if there exists a
sequence $K$ with $K(n)\le K(n+1)\le K(n)+1$ such that for
all $n\ge 0$
\begin{gather*}
F(n,0)<F(n,1)<F(n,2)<\dots <F(n,K(n)-1)\le F(n,K(n)),\\
F(n,K(n)) > F(n,K(n)+1)> \dots >F(n,n-1)>F(n,n)>F(n,n+1)=0.
\end{gather*}
\end{defi}

\begin{example}
If $f(n)=1$ for all $n\ge 0$, then 
\begin{equation*}
F(n,\ell)=\sum_{j=0}^n \binom{n-j}{\ell}=\binom{n+1}{\ell+1}
\end{equation*}
is unimodal. Setting $\ell=n-k$ and using $\binom{n-j}{k-j}=
\binom{n-j}{n-k}$ we may rewrite the sum as
\begin{equation*}
\sum_{j=0}^n \binom{n-j}{k-j}=\binom{n+1}{k}
\end{equation*} 
\end{example}

In general $F(n,\ell)$ will only be unimodal if we impose
a certain restriction on the growth of $f(n)$.
Before we give a criterion we note that
the recursion for the binomial coefficients implies the
following lemma:

\begin{lem}
Let $F(n,n+1)=0$. For $1\le \ell\le n$ it holds
\begin{align}\label{rec}
F(n+1,\ell) & = F(n,\ell)+F(n,\ell-1)\\
F(n+1,\ell+1)-F(n+1,\ell) & = F(n,\ell+1)-F(n,\ell-1)
\end{align}
\end{lem}

\begin{proof}

\begin{align*}
F(n,\ell)+F(n,\ell-1)
 & = \sum_{j=0}^{n}\left(\binom{n-j}{\ell}+\binom{n-j}{\ell-1}\right)f(j)\\
 & = \sum_{j=0}^{n}\binom{n+1-j}{\ell}f(j)\\
 & = F(n+1,\ell)
\end{align*}

Substituting $\ell+1$ for $\ell$ in $(1)$ yields $F(n+1,\ell+1)=F(n,\ell+1)+F(n,\ell)$ so that
the difference yields $(2)$.
\end{proof}  

\begin{prop}
Let $f$ be a sequence satisfying

\begin{itemize}
\item[(a)] $\; f(n)>0$ for all $n\ge 0$ and $f(3)\le 2f(0)+f(1)$.
\item[(b)] $\; f(n+1) \ge f(n)$ for all $n\ge 0$.
\item[(c)] $\; f(n) < \sum_{j=0}^{n-1}f(j)$ for all $n\ge 3$.
\end{itemize}        

Then $F(n,\ell)=\sum_{j=0}^n \binom{n-j}{\ell}f(j)$ is unimodal.
\end{prop}

\begin{proof}
The result follows by induction on $n$. For $n\le 3$ one directly obtains
$K(0)=K(1)=0$, $K(2)=0,1$ and $K(3)=1$ by $(a),(b),(c)$. 
For example, if $n=3$ then $F(3,0)\le F(3,1) >F(3,2)>F(3,3)>0$ says
$$f(0)+f(1)+f(2)+f(3)\le 3f(0)+2f(1)+f(2)>3f(0)+f(1)>f(0)>0$$
which follows from the assumptions.
Assuming for $n$ 
\begin{gather*}
F(n,0)<F(n,1)<F(n,2)<\dots <F(n,K(n)-1)\le F(n,K(n)),\\
F(n,K(n)) > F(n,K(n)+1)> \dots >F(n,n-1)>F(n,n)>F(n,n+1)=0.
\end{gather*}   
we obtain for $n+1$ using the recursion $(2)$:
\begin{gather*}
F(n+1,1)<F(n+1,2)<\dots <F(n+1,K(n)),\\
F(n+1,K(n)+1) > F(n+1,K(n)+1)> \dots >F(n+1,n)>F(n+1,n+1)>0.
\end{gather*}    
If $F(n+1,K(n))\le F(n+1,K(n)+1)$ we set $K(n+1)=K(n)+1$, and
otherwise $K(n+1)=K(n)$. It remains to show that $F(n+1,0)<F(n+1,1)$.
But since $F(n+1,0)=F(n,0)+f(n+1)$ and $K(n)\ge 1$ for  $n\ge 3$
we have 
\begin{align*}
F(n+1,1)-F(n+1,0) & = F(n,1)-f(n+1) \\
                  & \ge F(n,0)-f(n+1) \\
                  & =f(0)+f(1)+\dots +f(n)-f(n+1) > 0
\end{align*}
by assumption $(c)$.
\end{proof} 

\begin{cor}\label{uni}
\begin{equation*}
P(n,n-k)=\sum_{j=0}^{n}\binom{n-j}{n-k}p(j)=p(n,k)
\end{equation*}
is unimodal with $0\le k\le n$.      
\end{cor}

\begin{proof}
We can apply the proposition since the partition function $p(n)$ 
satisfies conditions $(a),(b),(c)$.
Here only $(c)$ is non-trivial. In fact, it is well known that
\begin{equation*}
p(n)\le p(n-1)+p(n-2)
\end{equation*}  
for all $n\ge 2$, i.e., that $p(n)$ is a "sub-Fibonacci" sequence.
If we set $\ell =n-k$, then $0\le k\le n$ and 
$P(n,n-k)$ is unimodal.
\end{proof}

\section{Lemmas on $p(n,k)$}

For the proof of the theorems we need some lemmas.

\begin{lem}\label{gen}
Denote by $p_k(j)$ the number of those partitions of $j$ in which each term
does not exceed $k$. If $\abs q<1$ then
\begin{align}\label{pro}
\sum_{j=0}^{\infty}p_k(j)q^j & =\prod_{j=1}^k\frac{1}{1-q^j}\\
\sum_{j=0}^{\infty}jp_k(j)q^j & =\sum_{j=1}^k\frac{jq^j}{1-q^j}\cdot \prod_{j=1}^k\frac{1}{1-q^j}
\end{align}
\end{lem}

\begin{proof}
The first identity is well known, see for example \cite{AN}, Theorem 13-1.
The product is the generating function for $p_k(n)$.
The second identity follows from the first by differentiation. 
Denoting $F_k(q)=\prod_{j=1}^k\frac{1}{1-q^j}$ we have
\begin{align*}
q\cdot \frac{d}{dq}F_k(q) & =\sum_{j=1}^k\frac{jq^j}{1-q^j}\cdot F_k(q)\\
\sum_{j=0}^{\infty}jp_k(j)q^j & = q\cdot \frac{d}{dq}\sum_{j=0}^{\infty}p_k(j)q^j 
\end{align*}
\end{proof}

In the following we will need good upper bounds for the 
infinite product 
\begin{equation*}
F(q):=\prod_{j=1}^{\infty}\frac{1}{1-q^j}.
\end{equation*}
$F(q)$ is directly related to the Dedekind eta-function, which is
defined on the upper half plane $\HH$ as 
\begin{equation*}
\eta (z):=q^{\frac{1}{24}}\prod_{j=1}^{\infty}(1-q^j),
\end{equation*}
where $q:=e^{2\pi iz}$. To obtain that approximately
$F(\frac{1}{2})=3.4627466194550636$, we could use 
$F(\frac{1}{2})=(\frac{1}{2})^{\frac{1}{24}}\cdot\eta (z)^{-1}$
with $z=\frac{i\log 2}{2\pi}$. The eta-function can be computed by many computer algebra systems.
On the other hand, it is not difficult to estimate the product
directly.

\begin{lem}\label{rest}
For $0<q<1$ and $\ell \ge 2$ we have
\begin{align}\label{ab}
\prod_{j=1}^{\infty}\frac{1}{1-q^j} & <\exp{\left(\frac{q^{\ell}}{(1-q)^2}\right)}\cdot 
\prod_{j=1}^{\ell -1}\frac{1}{1-q^j}\\ 
\sum_{j=1}^{\infty}\frac{jq^j}{1-q^j} & <\frac{q}{(1-q)^3}+
\sum_{j=1}^{\ell -1}\frac{jq^j(q^j-q)}{(1-q^j)(1-q)}
\end{align}
\end{lem}               

\begin{proof}
By the mean value theorem there exists a $\tau_j$ with 
$1-q\le 1-q^j<\tau_j <1$  such that $\log \frac{1}{1-q^j}=-\log(1-q^j)=\tau_j^{-1}q^j$
for all $j\ge 1$. Hence 
$$\log \frac{1}{1-q^j}<\frac{q^j}{1-q}$$
 for all $j\ge 1$ and
\begin{equation*}
\sum_{j=\ell}^{\infty} \log \frac{1}{1-q^j}< \sum_{j=\ell}^{\infty}\frac{q^j}{1-q}=\frac{q^{\ell}}{(1-q)^2}.
\end{equation*}
Taking exponentials on both sides yields 
$$\prod_{j=\ell}^{\infty}\frac{1}{1-q^j}<\exp{\frac{q^{\ell}}{(1-q)^2}}.$$
This proves \eqref{ab}. To show the second inequality we again use $1-q\le 1-q^j$ and 
\begin{equation*}
\sum_{j=1}^{\infty}\frac{jq^j}{1-q}=\frac{q}{(1-q)^3}
\end{equation*}
so that
\begin{equation*}
\sum_{j=1}^{\infty}\frac{jq^j}{1-q^j} = \sum_{j=1}^{\ell -1}\frac{jq^j}{1-q^j}+
\sum_{j=\ell}^{\infty}\frac{jq^j}{1-q^j}
< \sum_{j=1}^{\ell -1}\frac{jq^j}{1-q^j}-\sum_{j=1}^{\ell -1}\frac{jq^j}{1-q}+\frac{q}{(1-q)^3}
\end{equation*} 
\end{proof}

\begin{lem}\label{links}
For $n\ge 4$ and $k=\lfloor\frac{n+3}{2}\rfloor$ it holds
\begin{equation}\label{l1}
\sum_{j=0}^k (n+1-2k+j)\binom{n-j}{k-j}p(j)>0
\end{equation}
\end{lem}

\begin{proof}
If $n$ is even, then $k=\frac{n+2}{2}$ and $n+1-2k+j=j-1$.
Since for $j\ge 2$ all terms of the sum are positive, it is
enough to estimate the sum of the first $4$ terms:
\begin{align*}
\sum_{j=0}^3 (j-1)\binom{n-j}{k-j}p(j)& = \binom{n}{k}\sum_{j=0}^3 (j-1)\, a_{n,k,j}\, p(j)\\
 & =\binom{n}{k} \left(-1+2\cdot\frac{k(k-1)}{n(n-1)}+
6\cdot\frac{k(k-1)(k-2)}{n(n-1)(n-2)}\right)\\
& =  \binom{n}{k} \left(\frac{n+14}{4(n-1)}\right)>0
\end{align*}
for $k\ge 3$ and $n\ge 4$, where 
\begin{equation*}
a_{n,k,j}:=\binom{n-j}{k-j}\binom{n}{k}^{-1}= \frac{k!(n-j)!}{n!(k-j)!}
\end{equation*}
If $n$ is odd, then $k=\frac{n+3}{2}$ and $n+1-2k+j=j-2$.
Now the sum in \eqref{l1} contains two negative terms and 
one needs the first 8 terms:
\begin{equation*}
\sum_{j=0}^7 (j-2)\binom{n-j}{k-j}p(j) = \binom{n}{k} 
\left(\frac{5(11n^4+120n^3-2966n^2+9864n+10251)}
{128n(n-2)(n-4)(n-6)}\right)>0.
\end{equation*}  
for $k\ge 7$, $n\ge 11$. For $4\le n\le 10$ the sum in \eqref{l1} is also positive.  
\end{proof}

In the same way we obtain:

\begin{lem}\label{gr}
Let $\ell=\lfloor \frac{n+5}{2} \rfloor$.
For all $n\ge 4$ and all $k$ with $\ell\le k\le n$ it holds
\begin{equation*}
p(n,k)>\frac{1745}{512}\binom{n}{k}  
\end{equation*}
\end{lem}

\begin{proof}
Since $a_{n,k,j}\ge a_{n,\ell,j}$ for $k\ge \ell$ we have
for $n\ge 18$

\begin{equation*}
p(n,k)\ge \binom{n}{k}\sum_{j=0}^\ell a_{n,\ell,j}p(j)\ge 
\binom{n}{k} \sum_{j=0}^{11} a_{n,\ell,j}p(j)\ge \frac{1745}{512}\binom{n}{k}
\end{equation*}
%\begin{equation*}
%\frac{1745n^7-38722n^6+277352n^5-535240n^4-1030000n^3+704192n^2+5134848n-1935360
%}{512n(n-1)(n-2)(n-3)(n-5)(n-7)(n-9)}
%\end{equation*} 
For $4\le n\le 18$ the lemma is true also. 
\end{proof}

\begin{lem}\label{rechts}
For $n\ge 4$ and $k=\lfloor\frac{n+3}{2}\rfloor + 1$ it holds
\begin{equation}\label{l2}
\sum_{j=0}^k (n+1-2k+j)\binom{n-j}{k-j}p(j)<0
\end{equation}
\end{lem} 
\begin{proof}
If $n$ is even, then $k=\frac{n+4}{2}$.
We can rewrite \eqref{l2} as
\begin{equation}\label{ungl}
\sum_{j=0}^k j\binom{n-j}{k-j}p(j)<3 p(n,k)
\end{equation}
It can be checked by computer that the inequality holds
for small $n$. We have used the computer algebra package Pari to verify it.
So we may assume, let us say, $n\ge 500$. Then $\frac{k}{n}\le q=\frac{252}{500}$
for all $n\ge 500$ and hence
\begin{equation*}
a_{n,k,j}  =\frac{k}{n}\cdot \frac{k-1}{n-1}\cdots \frac{k-j+1}{n-j+1}\le \left(\frac{k}{n}\right)^j \le q^j.
\end{equation*}
It follows
\begin{align*} 
\sum_{j=0}^k j\binom{n-j}{k-j}p(j)& =\binom{n}{k}\sum_{j=0}^{k} j\, a_{n,k,j} \, p(j)  < 
\binom{n}{k}\sum_{j=0}^{k} j p(j)q^j \\
& < \binom{n}{k}\sum_{j=0}^{\infty} j p_k(j)q^j <9.96868 \binom{n}{k}
\end{align*} 
The last inequality follows from Lemma $\ref{gen}$ and Lemma $\ref{rest}$:
for $q=\frac{252}{500}$ we have
\begin{align*}
\prod_{j=1}^{\infty} \frac{1}{1-q^j} & < 3.54029829 \\
\sum_{j=1}^{\infty}\frac{jq^j}{1-q^j} & < 2.81577392
\end{align*}
and therefore
\begin{equation*}
\sum_{j=0}^{\infty} j p_k(j)q^j\le
\sum_{j=1}^{\infty}\frac{jq^j}{1-q^j}\cdot \prod_{j=1}^{\infty}
\frac{1}{1-q^j} < 9.96867959.
\end{equation*}
On the other hand we have by Lemma $\ref{gr}$
\begin{equation*}
3p(n,k)>\frac{5235}{512}\binom{n}{k}
\end{equation*}
so that inequality \eqref{ungl} follows. If $n$ is odd then $k=\frac{n+5}{2}$ and the
proof works as before.
\end{proof}

\begin{lem}
For $1\le k\le n-1$ it holds
\begin{equation}
p(n,k)<\binom{n}{k}\prod_{j=1}^{\infty} \frac{1}{1-(\frac{k}{n})^j}.
\end{equation}
\end{lem}

\begin{proof}
Estimating as in the preceding lemma and applying Lemma $\ref{gen}$ we have 
\begin{equation*}
p(n,k)\le\binom{n}{k}\sum_{j=0}^{k} \left(\frac{k}{n}\right)^j p_k(j)<
\binom{n}{k}\prod_{j=1}^{\infty} \frac{1}{1-(\frac{k}{n})^j}
\end{equation*}
\end{proof}   

\section{Proof of the Theorems} 

{\bf Proof of Theorem \ref{th2}:} Assume that
$n\ge 4$ is fixed. We have already proved that $p(n,k)$
is unimodal. 
What we must show is that
$p(n,k)$ becomes maximal exactly for $k=\lfloor \frac{n+3}{2} \rfloor$.
We formulate this as two lemmas:

\begin{lem}
For $n\ge 4$ and $1\le k\le \lfloor\frac{n+3}{2}\rfloor$ we have
\begin{equation*} 
p(n,k-1) < p(n,k) 
\end{equation*}
\end{lem}

\begin{proof}
Using $p(n+1,k)-p(n,k)=p(n,k-1)$ we see that we have to prove
$2p(n,k)-p(n+1,k)>0$ which is equivalent to the following 
inequality:

\begin{equation*}
\sum_{j=0}^k (n+1-2k+j) \binom{n-j}{k-j}p(j)>0.
\end{equation*} 

But this is obvious for $1\le k\le \lf \frac{n+3}{2}\rf -1$, because
in that case the sum has for $j\ge 1$ only positive terms and the first
term with $j=0$ is nonnegative.
For $k=\lf \frac{n+3}{2}\rf $ there exist negative terms, but the claim
follows from Lemma $\ref{links}$.  

\end{proof}

\begin{lem}
For $n\ge 4$ and $\lfloor\frac{n+3}{2}\rfloor +1 \le k \le n$ we have
\begin{equation}\label{down} 
p(n,k-1) > p(n,k)
\end{equation} 
\end{lem}

\begin{proof}

The inequality is equivalent to
\begin{equation*}
\sum_{j=0}^k (n+1-2k+j) \binom{n-j}{k-j}p(j)<0.
\end{equation*} 
For $k=\lfloor\frac{n+3}{2}\rfloor +1$ 
it follows from Lemma $\ref{rechts}$. 
We can now apply the unimodality of $p(n,k)$, see Corollary $\ref{uni}$, to
obtain the lemma.
\end{proof}

{\bf Proof of Theorem \ref{th3}:} For $n<500$ the theorem can be checked by
computer. Using Sterling's formula we obtain
\begin{equation*}
\binom{n}{\lfloor\frac{n+3}{2}\rfloor}< \frac{2^n}{\sqrt{\pi n/2}}
\end{equation*}    
for all $n\ge 1$ and hence with $q=252/500$, $k(n)=\lfloor\frac{n+3}{2}\rfloor$,
$n\ge 500$
\begin{equation*}
p(n,k)\le p(n,k(n))<\binom{n}{k(n)}
\prod_{j=1}^{\infty}\frac{1}{1-q^j}< 3.54029829\cdot\frac{2^n}{\sqrt{\pi n/2}}<
\frac{2.825}{\sqrt{n}}2^n.
\end{equation*}

For the proof of the propositions we need the following lemma.

\begin{lem}\label{13}
Let $\al=\sqrt{\frac{2}{3}}\pi$. Then for $n\ge 3$ we have
\begin{equation*}
\frac{\sqrt{n}}{\sqrt{n+1}-1}<
1+\frac{\pi}{\sqrt{6n}}<
e^{\al\sqrt{n}\left(\sqrt{1+\frac{1}{n}}-1\right)}
\end{equation*}
\end{lem}

\begin{proof}
Using the inequality
$$1+\frac{1}{2n}-\frac{1}{8n^2}<\sqrt{1+\frac{1}{n}}$$
and $\exp(x)>1+x+x^2/2$ for $x>0$
we obtain
\begin{align*}
e^{\al\sqrt{n}\left(\sqrt{1+\frac{1}{n}}-1\right)} & >
\exp \Bigl(\al\sqrt{n}\Bigl(\frac{1}{2n}-\frac{1}{8n^2}\Bigr)\Bigr) =
\exp \Bigl(\frac{\pi}{\sqrt{6n}}\Bigl(1-\frac{1}{4n}\Bigr)\Bigr)\\
 & > 1+\frac{\pi}{\sqrt{6n}}-\frac{\pi}{4n\sqrt{6n}}
+\frac{\pi^2}{12n}-\frac{\pi^2}{24n^2}
+\frac{\pi^2}{192n^3}\\
 & > 1+\frac{\pi}{\sqrt{6n}}
\end{align*}
for $n\ge 1$. On the other hand we have for $n\ge 17$
\begin{align*}
\frac{1}{1+\frac{\pi}{\sqrt{6n}}} & < 
1-\frac{\pi}{\sqrt{6n}}+\frac{\pi^2}{6n} < 1+\frac{1}{2n}-
\frac{1}{8n^2}-\frac{1}{\sqrt{n}}\\
 & < \sqrt{1+\frac{1}{n}}-\frac{1}{\sqrt{n}}=\frac{\sqrt{n+1}-1}{\sqrt{n}}
\end{align*}
Taking reciprocal values yields the second part of the lemma. For
$3\le n\le 16$ one verifies the lemma directly.
\end{proof} 

{\bf Proof of Proposition \ref{pr1}:} 
Let $\al=\sqrt{\frac{2}{3}}\pi$.
In \cite{APO}, section $14.7$ formula $(11)$, the following upper bound for
$p(n)$ is proved:
\begin{equation*}
p(n)<\frac{\pi}{\sqrt{6n}}e^{\al\sqrt{n}} \quad \text{for all} \;\; n\ge 1
\end{equation*}
We want to prove the proposition by induction on $n$.
By Lemma $\ref{13}$ we have 
\begin{equation*}
1+\frac{\pi}{\sqrt{6n}}<
e^{\al\sqrt{n+1}-\al\sqrt{n}}
\end{equation*}
which holds for all $n\ge 1$.
Assuming the claim for $n-1$ it follows for $n$:
\begin{align*}
p(n,n) =p(n-1,n-1)+p(n) & <e^{\al\sqrt{n}}+
\frac{\pi}{\sqrt{6n}}e^{\al\sqrt{n}} \\
         & = \left(1+\frac{\pi}{\sqrt{6n}}\right)e^{\al\sqrt{n}}<e^{\al\sqrt{n+1}}
\end{align*}

{\bf Proof of Proposition \ref{pr2}:} 
It follows from Lemma $\ref{13}$ that 
\begin{equation*}
\sqrt{n}e^{\al\sqrt{n}}<\left(\sqrt{n+1}-1\right)e^{\al\sqrt{n+1}}
\end{equation*}
By induction on $n$ and Proposition $\ref{pr1}$ we have:
\begin{align*}
p(n+1,n) =p(n,n)+p(n,n-1) & <e^{\al\sqrt{n+1}}+\sqrt{n}e^{\al\sqrt{n}}\\
         & <\sqrt{n+1}e^{\al\sqrt{n+1}}
\end{align*}

\end{document}